\documentclass[12pt]{amsart}
\usepackage{amsmath,amssymb,enumerate}
\allowdisplaybreaks
\newcommand{\IB}{{\mathbb B}}

\newcommand{\IK}{{\mathbb K}}

\newcommand{\IR}{{\mathbb R}}

\newcommand{\cD}{{\mathcal D}}
\newcommand{\cH}{{\mathcal H}}
\newcommand{\cI}{{\mathcal I}}
\newcommand{\cK}{{\mathcal K}}

\newcommand{\fn}{{\mathfrak n}}
\newcommand{\id}{\mathrm{id}}
\newcommand{\opp}{\mathrm{op}}
\newcommand{\uni}{\mathrm{u}}

\newcommand{\ip}[1]{\mathopen{\langle}#1\mathclose{\rangle}}
\newcommand{\circledT}{\mathbin{\hspace{.1em}\raisebox{.2ex}[1.2ex][0pt]{${\scriptscriptstyle\top}\hspace{-.6em}{\scriptstyle\bigcirc}$}}}
\newtheorem{thm}{Theorem}
\newtheorem{lem}[thm]{Lemma}
\newtheorem{prop}[thm]{Proposition}
\newtheorem{cor}[thm]{Corollary}
\theoremstyle{definition}
\newtheorem{defn}[thm]{Definition}

\title{Haagerup approximation property via bimodules}
\author{Rui Okayasu}
\address{Department of Mathematics Education, Osaka Kyoiku University,
Osaka \mbox{582-8582},
Japan}
\email{rui@cc.osaka-kyoiku.ac.jp}
\thanks{R.O.\ is partially supported by JSPS KAKENHI Grant Number 25800065}
\author{Narutaka Ozawa}
\address{RIMS, Kyoto University, Kyoto \mbox{606-8502}, Japan}
\email{narutaka@kurims.kyoto-u.ac.jp}
\thanks{N.O.\ is partially supported by JSPS KAKENHI Grant Number 26400114}
\author{Reiji Tomatsu}
\address{Department of Mathematics, Hokkaido University,
Hokkaido \mbox{060-0810},
Japan}
\email{tomatsu@math.sci.hokudai.ac.jp}
\thanks{R.T.\ is partially supported by JSPS KAKENHI Grant Number 24740095}

\begin{document}
\begin{abstract}
The Haagerup approximation property (HAP) is defined for finite von Neumann 
algebras in such a way that the group von Neumann algebra of a discrete group 
has the HAP if and only if the group itself has the Haagerup property. The HAP 
has been studied extensively for finite von Neumann algebras and it is recently 
generalized for arbitrary von Neumann algebras by Caspers--Skalski and 
Okayasu--Tomatsu. One of the motivations behind the generalization is the fact 
that quantum group von Neumann algebras are often infinite even though the 
Haagerup property has been defined successfully 
for locally compact quantum groups by Daws--Fima--Skalski--White. 
In this paper, we fill this gap by proving that the von Neumann algebra of 
a locally compact quantum group with the Haagerup property has the HAP. 
This is new even for genuine locally compact groups. 
\end{abstract}
\maketitle

\section{Introduction}

The notion of the Haagerup property for locally compact groups 
is introduced after the celebrated work of U. Haagerup (\cite{haagerup}) 
on the reduced group $\mathrm{C}^*$-algebras of the free groups. 
This notion is a very useful generalization of amenability and has been 
extensively studied in various settings (see \cite{ccjjv}). Like the case of 
amenability, it is only natural to capture this property through operator algebras. 
Indeed, M. Choda (\cite{choda}) has defined 
a property now called the Haagerup approximation property 
(we will abbreviate it as HAP) for \emph{finite tracial} von Neumann algebras and 
proved that the group von Neumann algebra $LG$ of a discrete group $G$ 
has the HAP if and only if $G$ has the Haagerup property. 
The HAP (or its relative version) has been exploited extensively in the study 
of finite von Neumann algebras as means of deformations in Popa's 
deformation-vs-rigidity strategy (\cite{popaicm}). 

The HAP is recently generalized for general von Neumann algebras 
independently by Caspers--Skalski \cite{cs1,cs2} and 
by Okayasu--Tomatsu \cite{ot1,ot2}. Their definitions vary, but 
turn out to be equivalent (see \cite{cost}) and seem to lay a satisfactory 
foundation for the study of the HAP for general von Neumann algebras. 
One of the motivations behind the generalization is the fact that  
quantum group von Neumann algebras are often infinite even though 
the Haagerup property has been defined successfully for locally compact 
quantum groups by Daws--Fima--Skalski--White (\cite{dfsw}). 
In this paper, we fill this gap by proving that the von Neumann algebra of 
a locally compact quantum groups with the Haagerup property has the HAP, 
and the converse also holds true for \emph{strongly inner amenable} 
locally compact quantum groups.
This extends the same result obtained by Daws--Fima--Skalski--White (\cite{dfsw}) 
for the case of discrete quantum groups. 
Another motivation would be to incorporate Popa's deformation-vs-rigidity 
strategy to the study of general von Neumann algebras. 

To pursue the latter motivation and to deal with locally compact (quantum) 
group von Neumann algebras, 
we take Connes's view (\cite[V.B]{connes}) that theory of bimodules is 
to von Neumann algebras what theory of unitary representations is to groups.
So, we will give yet another characterization of the HAP in terms of bimodules, 
which is preceded by the work of Bannon--Fang (\cite{bannon-fang}) 
for finite von Neumann algebras. 
Thus, we introduce the \emph{strict mixing} property for bimodules and 
prove that a von Neumann algebra has the HAP if and only if 
it admits a strictly mixing bimodule which is amenable and
that the von Neumann algebra of a locally compact quantum group 
with the Haagerup property admits such a bimodule. 

\subsection*{Conventions}
By $(M,\varphi)$, etc., we will mean a pair of von Neumann algebra $M$ and 
a distinguished fns (faithful normal semifinite) weight $\varphi$ on it. 
The symbol $\odot$ means the algebraic tensor product, while $\otimes$ means 
the von Neumann algebraic, Hilbert space, or the spatial $\mathrm{C}^*$-algebraic 
tensor product. All $*$-repre\-sen\-tations are assumed to be non-degenerate. 

\section{Preliminary on bimodules} 

In this section, we will review the theory of bimodules over von Neumann algebras. 
See \cite[V.B]{connes}, \cite{popa}, or \cite[XI.3]{takesakiII} for a comprehensive treatment. 
In literature, bimodules are also called \emph{correspondences}. 

Let $M$ and $N$ be von Neumann algebras. 
A Hilbert space $\cH$ is said to be an \emph{$M$-$N$ bimodule} if 
it comes together with a $*$-repre\-sen\-tation
$\pi_{\cH}$ of $M\odot N^{\opp}$ that is normal in each variable. 
Here $N^\opp$ denotes the opposite von Neumann algebra of $N$. 
We refer $\pi_{\cH}|_M$ as the left $M$-action and $\pi_{\cH}|_{N^{\opp}}$ 
as the right $N$-action, and simply write $a\xi x = \pi_{\cH}(a\otimes x^\opp)\xi$ 
for $a\in M$, $x\in N$, and $\xi\in\cH$. 
The complex conjugate $\bar{\cH}$ of an $M$-$N$ bimodule is naturally 
an $N$-$M$ bimodule. 
The notation ${}_M\cH_{N}$ will indicate that $\cH$ is an $M$-$N$ bimodule.

An $M$-$N$ bimodule $\cH$ is \emph{weakly contained} 
in another $M$-$N$ bimodule $\cK$ (denoted by $\cH\preceq\cK$) if the 
identity map on $M\odot N^{\opp}$ extends to a continuous 
$*$-homo\-mor\-phism from $\mathrm{C}^*(\pi_{\cK}(M\odot N^{\opp}))$ 
to $\mathrm{C}^*(\pi_{\cH}(M\odot N^{\opp}))$, 
that is to say, if for any $\xi\in\cH$, 
any finite subsets $E\subset M$ and  $F\subset N$, and any $\varepsilon>0$, 
one can find $\eta_1,\ldots,\eta_n\in\cK$ such that 
$|\ip{ a\xi x, \xi} - \sum_i\ip{ a\eta_i x,\eta_i}| < \varepsilon$ 
for all $(a,x)\in E\times F$.

The identity bimodule over $M$ is the $M$-$M$ bimodule $L^2(M)$, given by  
$a\xi x = aJx^*J \xi$, where $L^2(M)$ is the standard form for $M$ and $J$ 
is the modular conjugation. 
When an fns weight $\varphi$ on $M$ is fixed, we identify 
$L^2(M)$ with $L^2(M,\varphi)$ and $J$ with $J_\varphi$. 
The vector in $L^2(M)$ that corresponds to 
$x\in \fn_\varphi:=\{ x \in M : \varphi(x^*x)<\infty\}$ is 
denoted by $x\varphi^{1/2}$. Also, $\varphi^{1/2}x := Jx^*\varphi^{1/2}$ 
for $x\in\fn_\varphi^*$. 

For $M$-$N$ bimodules (or just right $N$-modules) $\cH$ and $\cK$, 
the Banach space of bounded right $N$-module maps from $\cH$ into $\cK$ 
is denoted by $\IB(\cH_N,\cK_N)$. In case $\cH$ and $\cK$ coincide, 
we simply denote it by $\IB(\cH_N)$. 
Thus, $\IB(L^2(N)_N)$ coincides with $N$ acting on $L^2(N)$ from the left. 

Let us fix fns weights $\varphi$ on $M$ and $\psi$ on $N$. 
A vector $\xi\in {}_M\cH_N$ is said to be \emph{left $\psi$-bounded} if 
$L_\psi(\xi)\colon \psi^{1/2}x \mapsto \xi x$, $x\in\fn_\psi^*$, 
is bounded and hence defines an element in $\IB(L^2(N,\psi)_N,\cH_N)$. 
The subspace $\cD(\cH,\psi)$ of left $\psi$-bounded vectors is 
dense in $\cH$ (\cite[Lemma IX.3.3]{takesakiII}). 
We note that if $\xi\in\cD(\cH,\psi)$ and $a\in M$, then 
$a\xi\in\cD(\cH,\psi)$ and $L_\psi(a\xi) = aL_\psi(\xi)$.  
For $\xi_1,\xi_2\in\cD(\cH,\psi)$, we denote by 
$L_\psi(\xi_2^*\times\xi_1)$ the element in $N$ that corresponds to 
$L_\psi(\xi_2)^*L_\psi(\xi_1) \in \IB(L^2(N)_N)=N$.
It in fact belongs to $\fn_\psi^*\fn_\psi$ and satisfies 
$\psi(L_\psi(\xi_2^*\times\xi_1)) = \ip{\xi_1,\xi_2}$ 
(see \cite[p.199]{takesakiII}).
Similarly, $\xi$ is said to be \emph{right $\varphi$-bounded} 
if $R_\varphi(\xi)\colon a\varphi^{1/2} \mapsto a\xi$, 
$a\in\fn_\varphi$, is bounded. It is \emph{$(\varphi,\psi)$-bounded} (or simply bounded)
if it is simultaneously left $\psi$- and right $\varphi$-bounded. 
The bounded vectors are dense in $\cH$ 
(see Theorem~\ref{thm:ambibounded}).

Let $\theta\colon M\to N$ be a normal completely positive map. 
Associated to it is the $M$-$N$ bimodule $L^2(M\otimes_{\theta} N)$ which is defined to be 
the Hilbert space completion of $M\odot L^2(N)$ under 
the semi-inner product $\ip{ b_1\otimes_\theta \eta_1, b_2\otimes_\theta \eta_2} 
= \ip{\theta(b_2^*b_1)\eta_1,\eta_2}$ on simple tensors. 
Here the symbol $\otimes_\theta$ is used for mnemonic reason. 
The left $M$-action and the right $N$-action are given by 
$a(b\otimes_\theta\eta)x = (ab)\otimes_\theta(\eta x)$.
Suppose now that $\psi\circ\theta \le C\varphi$ for some constant $C>0$, 
and fix analytic elements $b\in\fn_\varphi$ and $y\in\fn_\psi$.
Then, the vector $\xi=b\otimes_\theta (y\psi^{1/2})$ is $(\varphi,\psi)$-bounded. 
Indeed, one has $\| \xi x\| \le \|y^*\theta(b^*b)y\|^{1/2} \|\psi^{1/2}x\|$ and
\begin{align*}
\| a\xi \| &= \|\theta(b^*a^*ab)^{1/2}\psi^{1/2} \sigma^\psi_{i/2}(y)\| \\
 &\le C^{1/2}\|\sigma^\psi_{i/2}(y)\| \| ab\varphi^{1/2}\| 
 \le C^{1/2}\|\sigma^\varphi_{i/2}(b)\| \|\sigma^\psi_{i/2}(y)\| \| a\varphi^{1/2}\|. 
\end{align*}

Let $M$, $N$, and $P$ be von Neumann algebras. The the \emph{relative tensor product} 
${}_M\cH_N{\otimes}_N\cK_P$ of the bimodules ${}_M\cH_N$ and ${}_N\cK_P$ is the 
$M$-$P$ bimodule which is defined to 
be the Hilbert space completion of $\cD(\cH,\psi) \odot \cK$ under the semi-inner product 
$\ip{ \xi_1\otimes_\psi\eta_1,\xi_2\otimes_\psi\eta_2}
 = \ip{ L_\psi(\xi_2^*\times\xi_1) \eta_1,\eta_2}$
on simple tensors. 
If moreover $\eta_i$'s are right $\psi$-bounded, then one also has 
$\ip{ \xi_1\otimes_\psi\eta_1,\xi_2\otimes_\psi\eta_2}
 = \ip{\xi_1,\xi_2 \bar{R}_\psi(\eta_2^*\times\eta_1)}$, where 
$\bar{R}_\psi(\eta_2^*\times\eta_1)=J_\psi R_\psi(\eta_2)^* R_\psi(\eta_1)J_\psi \in N$ 
(\cite[Proposition 3.15]{takesakiII}).
The left $M$-action and the right $P$-action are given 
by $a(\xi\otimes_\psi\eta)x = (a\xi)\otimes_\psi(\eta x)$.
We note that if $\xi$ and $\eta$ are left $\psi$- and left $\omega$-bounded 
(here $\omega$ is an fns weight on $P$), 
then $\xi\otimes_\psi\eta$ is left $\omega$-bounded with 
$\|L_\omega( \xi\otimes_\psi\eta )\|\le \|L_\psi(\xi)\| \|L_\omega(\eta)\|$. 
Also, if $\xi$ and $\eta$ are $(\varphi,\psi)$- and $(\psi,\omega)$-bounded, 
then $\xi\otimes_\psi\eta$ is $(\varphi,\omega)$-bounded with 
$\|R_\varphi(\xi\otimes_\psi\eta )\| \le \|R_\varphi(\xi)\|\| R_\psi(\eta)\|$.
The relative tensor product construction is associative 
and continuous with respect to the weak containments, i.e., 
${}_N\cK_P\preceq{}_N\cK'_P$ implies 
${}_M\cH_N{\otimes}_N\cK_P \preceq {}_M\cH_N{\otimes}_N\cK'_P$, 
and likewise for the first variable. 
We note that there is a canonical isomorphisms ${}_M\cH_N{\otimes}_NL^2(N)_N\cong{}_M\cH_N$ 
via $\xi\otimes_\psi \psi^{1/2}x \leftrightarrow \xi x$ for  
$\xi\in\cD(\cH,\psi)$ and $x\in\fn_\psi^*$. 
Likewise ${}_ML^2(M)_M{\otimes}_M\cH_N\cong{}_M\cH_N$ via 
$a\varphi^{1/2}\otimes_\varphi \xi \leftrightarrow a\xi$ for 
$a\in\fn_\varphi$ and $\xi\in\cH$. 

The following fact is well-known in the case of finite von Neumann 
algebras and the general case is probably also known 
to the specialists, but the authors did not find it in literature. 

\begin{thm}[cf.\ {\cite[1.2.2]{popa}}]\label{thm:ambibounded}
Let $(M,\varphi)$ and $(N,\psi)$ be von Neumann algebras and ${}_M\cH_N$ be an $M$-$N$ bimodule.
Then the $(\varphi,\psi)$-bounded vectors are dense in $\cH$. 
\end{thm}
\begin{proof}
We first prove the theorem assuming that $M$ is semifinite 
(but $\varphi$ need not be a trace). 
For this, we claim that for every $\xi\in\cH$ there is a net $(c_i)_i$ 
of contractions in $M$ such that $c_i\xi$ are right $\varphi$-bounded and 
$c_i\xi \to \xi$. 
Indeed, fix an fns trace $\tau$ on $M$ and 
view normal states on $M$ as $\tau$-measurable operators on $L^2(M,\tau)$ 
affiliated with $M$ (see \cite[XI.2]{takesakiII} for an account of measurable operators). 
Thus the vector functional $\omega_\xi$ on $M$ corresponds to 
$h\in L^1(M,\tau)_+$ in such a way that $\ip{a\xi,\xi}=\tau(h a)$ 
for $a\in M$. Similarly, take an increasing net $(\varphi_j)_j$ of normal positive 
functionals on $M$ such that $\varphi=\sup\varphi_j$, and denote by $k_j$ the 
element in $L^1(M,\tau)_+$ that corresponds to $\varphi_j$. 
Let $c_{n,j}=\chi_{[n^{-1},\infty)}(k_j)(1+n^{-1}h)^{-1/2} \in M$. 
Since $\varphi$ is faithful, one has $c_{n,j}\to1$ ultrastrongly. 
Moreover, the inequality 
\[
c_{n,j} h c_{n,j}^* 
 = \chi_{[n^{-1},\infty)}(k_j) \frac{h}{1+n^{-1}h} \chi_{[n^{-1},\infty)}(k_j) 
 \le n^2 k_j 
\]
implies that $\|ac_{n,j}\xi\|=\tau(hc_{n,j}^*a^*ac_{n,j})^{1/2}\le n\varphi_j(a^*a)^{1/2}\le n\| a\varphi^{1/2}\|$ 
for every $a\in \fn_\varphi$. 
Thus we obtain the claim. 
Since the space $\cD(\cH,\psi)$ of left $\psi$-bounded vectors is dense 
and is a left $M$-module, we see by the claim that the bounded vectors 
are dense in $\cH$. 

Now let $(M,\varphi)$ be an arbitrary von Neumann algebra 
and denote the modular action by $\sigma$. 
Let $\tilde{M}=M\rtimes_\sigma\IR$ and $\tilde{\varphi}$ 
denote the corresponding crossed product and the dual weight.
We will denote by $\pi\colon M\to\tilde{M}$ the canonical inclusion and 
by $\{u_s : s\in\IR\}$ the unitary elements in $\tilde{M}$ that implement 
the modular action. Thus $u_s\pi(a)u_s^*=\pi(\sigma_s(a))$ for $a\in M$ and
$\tilde{M} = (\pi(M) \cup u_{\IR})''$. 
Since $\tilde{M}$ is semifinite, the result of the previous paragraph says that 
$(\tilde{\varphi},\psi)$-bounded vectors are dense in the relative tensor product 
$\tilde{\cH}:={}_{\tilde{M}}L^2(\tilde{M},\tilde{\varphi})_M{\otimes}_M\cH_N$.
We identify $\tilde{\cH}$ with $L^2(\IR,\cH)$ 
where the left $\tilde{M}$-action is given by 
$(u_s\xi)(t)=\xi(t-s)$ for $s\in\IR$ and $(\pi(a)\xi)(t) = \sigma_t^{-1}(a)\xi(t)$ 
for $a\in M$; and the right $N$-action is given by $(\xi x)(t)=\xi(t)x$ for $x\in N$. 
For every $f \in K(\IR)$ (compactly supported continuous functions), 
we define $L_f\in\IB(\cH_N, L^2(\IR,\cH)_N)$ by 
$L_f\xi=f\otimes\xi$ where $(f\otimes\xi)(t)=f(t)\xi$.
Since $\{ f^**g : f,g\in K(\IR) \}$ has dense span in $L^2(\IR)$, 
the proof of the theorem will be done once we prove 
that $L_{f^**g}^* \zeta \in \cH$ is $(\varphi,\psi)$-bounded
for every $f,g\in K(\IR)$ and every $(\tilde{\varphi},\psi)$-bounded $\zeta\in\tilde{\cH}$. 
That $L_{f^**g}^* \zeta$ is left $\psi$-bounded is obvious. 
A direct computation shows that $a L_{f^**g}^*\zeta = L_g^*\tilde{a}\zeta$ 
for every $a\in M$. Here $\tilde{a} = \int f(s)au_s\,ds \in \tilde{M}$. 
Indeed, for every $\xi\in\cH$ one has 
\begin{align*}
\ip{ a L_{f^**g}^*\zeta, \xi} 
 &= \int \ip{\zeta(t), \int \bar{f}(s)g(t+s)\,ds\, a^*\xi} \,dt\\
 &= \iint f(s) \ip{a\zeta(t-s),g(t)\xi} \,dt\,ds 
 = \ip{\tilde{a}\zeta,L_g\xi}.
\end{align*}
It follows that 
$\|a L_{f^**g}^*\zeta\|
 \le \|L_g\| \|R_{\tilde{\varphi}}(\zeta)\| \| \tilde{a}\tilde{\varphi}^{1/2}\| 
 = \|g\|_2 \|R_{\tilde{\varphi}}(\zeta)\| \|f\|_2 \|a\varphi^{1/2}\|$
for every $a\in\fn_\varphi$, and hence $L_{f^**g}^*\zeta$ is right $\varphi$-bounded also.
\end{proof}

\begin{lem}\label{lem:l2op}
Let $(M,\varphi)$ and $(N,\psi)$ be von Neumann algebras 
and $\xi \in {}_M\cH_N$ be a $(\varphi,\psi)$-bounded vector. 
Then, the completely positive map 
\[
\theta_\xi\colon M\ni a\mapsto L_\psi(\xi)^*aL_\psi(\xi) = L_\psi(\xi^*\times a\xi) \in N
\]
satisfies $\psi\circ\theta_\xi \le \|R_\varphi(\xi)\|^2\varphi$. 
The corresponding operator 
\[
T_{\varphi,\psi}(\xi) \colon L^2(M,\varphi)\ni a\varphi^{1/2} \mapsto \theta_\xi(a)\psi^{1/2}\in L^2(N,\psi), 
\quad a\in\fn_\varphi,
\]
is equal to $L_\psi(\xi)^*R_\varphi(\xi)$. 
Moreover, for every $a\in\fn_\varphi$ and $x\in\fn_\psi$, one has 
\[
\ip{\theta_\xi(a)\psi^{1/2},\psi^{1/2}x^*} = \ip{ a\xi,\xi x^*}
 = \ip{ L_\psi(\xi)^*R_\varphi(\xi) a\varphi^{1/2}, \psi^{1/2}x^*}.
\]
In particular, the $M$-$N$ bimodules 
$\overline{M\xi N}$ and $L^2(M\otimes_{\theta_\xi}N)$ are isomorphic 
via the correspondence 
$a\xi x \leftrightarrow a\otimes_{\theta_\xi}\psi^{1/2}x$ for 
$a\in M$ and $x\in \fn_\psi^*$. 
\end{lem}
\begin{proof}
That the map is completely positive is clear. 
For $a\in M$, one has
\[
\psi(\theta_\xi(a^*a))=\psi(L_\psi((a\xi)^*\times(a\xi))) = \|a\xi\|^2
 \le \| R_\varphi(\xi)\|^2\varphi(a^*a).
\]
This proves the first assertion. 
It follows that for every $a\in\fn_\varphi$ and $x,y\in\fn_\psi$, one has 
$\theta_\xi(a)\in\fn_\psi$ and 
\[
\ip{\theta_\xi(a) \psi^{1/2}, \psi^{1/2}x^*y}
 = \ip{a\xi y^*,\xi x^*} = \ip{a\xi,\xi x^*y}
 = \ip{R_\varphi(\xi) a\varphi^{1/2},L_\psi(\xi)\psi^{1/2}x^*y}.
\]
Note that $\fn_\psi^*\fn_\psi$ is $L^2$-dense in $\fn_\psi$. 
The proof of the last assertion is routine. 
\end{proof}

Hence, every $M$-$N$ bimodule is isomorphic to a direct sum of 
bimodules of the form $L^2(M\otimes_\theta N)$. 
To complete the picture, recall from a previous paragraph that 
if $\psi\circ\theta \le C\varphi$ for some constant $C>0$ 
and $b\in\fn_\varphi$ and $y\in\fn_\psi$ are analytic elements, 
then $\xi=b\otimes_\theta (y\psi^{1/2})$ is $(\varphi,\psi)$-bounded. 
The completely positive map $\theta_\xi$ arising from Lemma~\ref{lem:l2op} 
is related to the original $\theta$ by the relation 
$\theta_\xi(a) = y^*\theta(b^*ab) y$. 
Indeed, one has 
\[
\ip{\theta_\xi(a) \psi^{1/2} x_1, \psi^{1/2} x_2} 
 = \ip{a \xi x_1,\xi x_2} = \ip{\theta(b^*ab)y\psi^{1/2} x_1,y\psi^{1/2} x_2}
\]
for every $x_1,x_2\in\fn_\psi^*$. Also, it is not difficult to see that 
the maps arising in Lemma~\ref{lem:l2op} are compatible with the relative 
tensor product, as follows.

\begin{lem}\label{lem:composition}
Let $(M,\varphi)$, $(N,\psi)$, and $(P,\omega)$ be von Neumann algebras, 
and $\xi \in {}_M\cH_N$ and $\eta\in {}_N\cK_P$ be bounded vectors. 
Then, 
$\theta_{\xi\otimes_\psi\eta}=\theta_\eta\circ\theta_\xi$ and 
$T_{\varphi,\omega}(\xi\otimes_\psi\eta)=T_{\psi,\omega}(\eta) T_{\varphi,\psi}(\xi)$. 
\end{lem}

\section{Mixing bimodules and the Haagerup approximation property}

Recall that a von Neumann algebra $M$ is amenable (or injective) 
if and only if the  identity bimodule  ${}_ML^2(M)_M$ is weakly contained in 
the \emph{coarse bimodule} ${}_ML^2(M)\otimes L^2(M)_M$.
The coarse bimodule has a very strong mixing property. 
The precise notion of mixing property for bimodules has been introduced 
in \cite{bannon-fang} in the case of finite von Neumann algebras 
under the name of ``$C_0$-correspondences.'' 
There it is proved that a finite von Neumann algebra $M$ has 
the HAP if and only if the identity correspondence $L^2(M)$ is weakly 
contained in a $C_0$-correspondence (\cite[Theorem 3.4]{bannon-fang}). 
In this section, we extend this to the general setting. 

\begin{defn}\label{defn:mix}
Let $(M,\varphi)$ and $(N,\psi)$ be von Neumann algebras 
with distinguished fns weights. 
An $M$-$N$ bimodule $\cH$ is said to be \emph{strictly mixing} if 
the family 
\[
\cH_{\mathrm{mix}} 
= \{ \xi\in\cH : \mbox{$\xi$ is $(\varphi,\psi)$-bounded 
and $T_{\varphi,\psi}(\xi)$ is compact}\}
\]
is $M$-$N$ cyclic in $\cH$ (i.e., $M\cH_{\mathrm{mix}}N$ has a dense span in $\cH$). 
As we will see in Corollary~\ref{cor:indep}, the strict mixing property of $\cH$ 
does not depend on the choices of fns weights (but the subset 
$\cH_{\mathrm{mix}}$ does).
\end{defn}

It is not clear whether $\cH_{\mathrm{mix}}$ is always a linear subspace, or 
at least it contains a linear subspace with the same closed linear span. 
However, the infinite multiple of $\cH$ has the latter property, as we will prove below. 
Hence, we may assume that $\cH$ has this property in all cases which are 
dealt in this paper. 

\begin{lem}[cf.\ {\cite[Definition 3.1]{bannon-fang}}]\label{lem:bfdefn}
Let $\theta\colon M\to N$ be a normal completely positive map 
such that $\psi\circ\theta\le C\varphi$ for some constant 
$C>0$ and such that $T_\theta\colon L^2(M,\varphi) \to L^2(N,\psi)$ 
defined by $T_\theta a\varphi^{1/2} = \theta(a)\psi^{1/2}$ 
is compact. Then, the $M$-$N$ bimodule $L^2(M\otimes_\theta N)$ 
is strictly mixing. 
Conversely, if $\cH$ is a strictly mixing $M$-$N$ bimodule such that 
${}_M\cH_N\cong\bigoplus_\kappa{}_M\cH_N$ for the density character $\kappa$ of $\cH$, 
then $\cH$ is isomorphic to a direct sum $\bigoplus_i L^2(M\otimes_{\theta_i} N)$ of 
bimodules as above. In particular, the subset $\cH_{\mathrm{mix}}$ contains a dense 
linear subspace of $\cH$. 
\end{lem}
\begin{proof}
We will prove that $L^2(M\otimes_\theta N)_{\mathrm{mix}}$ contains 
a dense linear subspace. 
We stick to the notations used in the previous section and put 
$\xi = \sum_k b_k\otimes_\theta (y_k\psi^{1/2})$ for finite sequences 
$b_k\in\fn_\varphi$, $y_k\in\fn_\psi$ of analytic elements. 
Then, by the remark following Lemma~\ref{lem:l2op}, one has 
\begin{align*}
\theta_\xi(a)\psi^{1/2}
 &= \sum_{k,l} y_k^* \theta(b_k^*ab_l) y_l\psi^{1/2} 
 = \sum_{k,l} y_k^*J_\psi\sigma^\psi_{i/2}(y_l)^*J_\psi\theta(b_k^*ab_l)\psi^{1/2} \\
 &= \sum_{k,l} y_k^*J_\psi\sigma^\psi_{i/2}(y_l)^*J_\psi T_\theta b_k^*J_\varphi\sigma^\varphi_{i/2}(b_l)^*J_\varphi a\varphi^{1/2}.
\end{align*}
Since $T_\theta$ is compact, the operator 
$T_{\varphi,\psi}(\xi) =\sum_{k,l} y_k^*J_\psi\sigma^\psi_{i/2}(y_l)^*J_\psi T_\theta b_k^*J_\varphi\sigma^\varphi_{i/2}(b_l)^*J_\varphi$ 
is also. 

Now let $\{\xi_i : i\in\kappa\}$ be an $M$-$N$ cyclic family in $\cH_{\mathrm{mix}}$ and 
$\cK_i:=\overline{M\xi_i N}\subset\cH$. 
Then, $\cH$ is isomorphic to a subbimodule of $\bigoplus\cK_i$, and 
$\cK_i\cong L^2(M\otimes_{\theta_{\xi_i}} N)$ for each $i$ by Lemma~\ref{lem:l2op}. 
Since $\cH$ has infinite multiplicity, $\cH\cong\cH\oplus\cK_i$ for each $i$ 
and $\cH\oplus\bigoplus_i\cK_i\cong\bigoplus_i\cK_i$. 
It follows that $\cH\cong\bigoplus_i(\cH\oplus\cK_i)\cong\bigoplus_i\cK_i$. 
Note that for any finite sequence $\xi_i\in(\cK_i)_{\mathrm{mix}}$, one has 
$\sum_i\xi_i \in (\bigoplus_i\cK_i)_{\mathrm{mix}}$, since $T_{\varphi,\psi}(\sum_i\xi_i)=\sum_iT_{\varphi,\psi}(\xi_i)$. 
\end{proof}

Although it will not be used, we note the following fact. The converse is not clear. 
\begin{lem}[cf.\ {\cite[Definition 2.3]{peterson-sinclair}}]
A strictly mixing $M$-$N$ bimodule $\cH$ is mixing in the sense that 
$\ip{ a_n\xi x_n,\xi}\to0$ for any $\xi\in\cH$ and any bounded nets 
$(a_n)_n$ in $M$ and $(x_n)_n$ in $N$, one of which is ultraweakly null. 
\end{lem}
\begin{proof}
Since $(a_n)_n$ and $(x_n)_n$ are bounded nets, 
we may assume that $\xi\in\cH_{\mathrm{mix}}$. 
Given $\varepsilon>0$, take $p\in\fn_\varphi$ and $q\in\fn_\psi$ 
such that $\|p\xi-\xi\|+\|\xi-\xi q^*\|<\varepsilon$. 
Then, $(a_np\varphi^{1/2})_n$ and $(\psi^{1/2} q^*x_n^*)_n$ are 
bounded nets respectively in $L^2(M)$ and $L^2(N)$, one of which is weakly null. 
Hence $\ip{a_np\xi x_n,\xi q^*}=\ip{T_{\varphi,\psi}(\xi)a_np\varphi^{1/2},\psi^{1/2} q^*x_n^*}\to0$, 
by the strict mixing property. 
Since $\varepsilon>0$ was arbitrary, one concludes that $\ip{ a_n\xi x_n,\xi}\to0$.
\end{proof}

\begin{prop}\label{prop:mixingtensor}
Let $(M,\varphi)$,  $(N,\psi)$, and $(P,\omega)$ be 
von Neumann algebras. 
Then the relative tensor product bimodule ${}_M\cH_N{\otimes}_N\cK_P$ is 
strictly mixing if one of ${}_M\cH_N$ and ${}_N\cK_P$ is strictly mixing. 
\end{prop}
\begin{proof}
This follows from Theorem~\ref{thm:ambibounded} and Lemma~\ref{lem:composition}.
\end{proof}

\begin{cor}\label{cor:indep}
The strict mixing property of an $M$-$N$ bimodule $\cH$ 
does not depend on the choices of fns weights. 
\end{cor}
\begin{proof}
Let $\cH$ be a strictly mixing 
$(M,\varphi)$-$(N,\psi)$ bimodule, and let $\varphi_0$ 
be another fns weight on $M$.
Let us view $L^2(M)$ as 
an $(M,\varphi_0)$-$(M,\varphi)$ bimodule.
Then by Proposition~\ref{prop:mixingtensor}, 
the relative tensor product ${}_ML^2(M)_M{\otimes}_M\cH_N$ 
is a strictly mixing $(M,\varphi_0)$-$(N,\psi)$ bimodule, but 
it is canonically isomorphic to $\cH$ as an $M$-$N$ bimodule. 
So the strict mixing property does not depend on the choice 
of $\varphi$, and similarly neither on $\psi$. 
\end{proof}

There are several equivalent definitions of 
the Haagerup approximation property (HAP) for a von Neumann algebra $M$, 
but in the end they are equivalent to that the finite 
von Neumann algebra $p(M\rtimes_\sigma\IR)p$ has the HAP 
for a modular action $\sigma$ and a finite projection $p$ with 
full central support (see \cite{cost,cs1,cs2,jolissaint,ot1,ot2}). 
We recall that such $p(M\rtimes_\sigma\IR)p$ is \emph{amenably equivalent} 
(in the sense of Anantharaman-Delaroche \cite{delaroche}) 
to the original $M$. 
Here we say an $M$-$N$ bimodule $\cH$ is \emph{left amenable} 
if the identity bimodule  ${}_ML^2(M)_M$ is weakly contained in 
${}_M\cH_N{\otimes}_N\bar{\cH}_M$;
and $M$ and $N$ are amenably equivalent if there exist a left amenable $M$-$N$ bimodule 
and a left amenable $N$-$M$ bimodule. See \cite{delaroche} for the details. 
In fact, it is not too difficult to see that, 
for any crossed product $\tilde{M}=M\rtimes_\sigma G$ of a von Neumann algebra $M$ 
by an amenable locally compact group $G$, 
both ${}_ML^2(\tilde{M})_{\tilde{M}}$ and ${}_{\tilde{M}}L^2(\tilde{M})_M$ are left amenable. 
We note that the relative tensor product of two left amenable bimodules 
is again left amenable by continuity (\cite[2.13]{delaroche}). 
The following theorem is reminiscent of the fact (\cite{el}) that 
a von Neumann algebra $M$ is \emph{semi-discrete} 
if and only if it is \emph{amenable}: ${}_ML^2(M)_M\preceq{}_ML^2(M)\otimes L^2(M)_M$.

\begin{thm}\label{thm:bihap}
For a von Neumann algebra $M$, the following are equivalent. 
\begin{enumerate}[(1)]
\item\label{con:b1}
$M$ has the HAP. 
\item\label{con:b2}
The identity bimodule $L^2(M)$ is weakly contained in 
a strictly mixing $M$-$M$ bimodule.
\item\label{con:b3}
There are a von Neumann algebra $N$ and a strictly mixing $M$-$N$ bimodule 
which is left amenable. 
\end{enumerate}
\end{thm}
\begin{proof}
Proposition~\ref{prop:mixingtensor} implies 
the equivalence $(\ref{con:b2})\Leftrightarrow(\ref{con:b3})$, as well as invariance 
of $(\ref{con:b3})$ under the amenable equivalence. 
Now the proof of $(\ref{con:b1})\Leftrightarrow(\ref{con:b3})$ reduces to the case 
of finite von Neumann algebras, but in which case it is done by \cite[Theorem 3.4]{bannon-fang}. 
\end{proof}
%

\section{Locally compact quantum group von Neumann algebras}

In this section, we study the relationship between the unitary representations 
of a locally compact quantum group $G$ and the bimodules of the group 
von Neumann algebra $LG$. 

For a locally compact group $G$, there are the function algebras 
$C_0(G)\subset L^\infty(G)$ and the group operator algebras 
$\mathrm{C}^*_\lambda(G)\subset LG$. 
A \emph{locally compact quantum group} $G$ also comes with 
pairs of a ultraweakly dense $\mathrm{C}^*$-subalgebra 
of a von Neumann algebra,
$C_0(G)\subset L^\infty(G)$ and $\mathrm{C}^*_\lambda(G)\subset LG$. 
The latter is often written as $C_0(\hat{G})\subset L^\infty(\hat{G})$.
See \cite{kustermans,kv} for a comprehensive treatment of theory 
of locally compact quantum groups.
We denote by $W\in M(C_0(G)\otimes \mathrm{C}^*_\lambda(G))$ 
the multiplicative unitary of $G$,
by $\varphi$ the left Haar weight (which is unique up to 
a scalar multiple) on $L^\infty(G)$, and by $\hat{\varphi}$ 
the dual weight on $LG$. 
The modular conjugations are written respectively by $J$ and $\hat{J}$. 
The map $\hat{R}\colon a\mapsto Ja^*J$ defines 
an anti-$*$-auto\-mor\-phism on $\mathrm{C}^*_\lambda(G)$ and is 
called the \emph{unitary antipode}  (\cite[5.3]{kustermans}).

The map $\lambda$ from $L^1(G) := L^\infty(G)_*$ 
to $\mathrm{C}^*_\lambda(G)$, given by 
$\lambda(\omega) = (\omega\otimes\id)(W)$, 
has a dense range and is called the \emph{left regular representation}. 
More generally, a \emph{unitary representation} of $G$ 
(or a \emph{unitary corepresentation} of $C_0(G)$) on $\cH_U$ 
is a unitary element $U\in M(C_0(G)\otimes\IK(\cH_U))$ 
which satisfies $(\Delta\otimes\id)(U)=U_{13}U_{23}$ (see \cite[Definition 3.5]{kustermans}). 
There exist the \emph{universal group $\mathrm{C}^*$-algebra} $\mathrm{C}^*_{\uni}(G)$ 
and the universal unitary representation $W_{\uni} \in M(C_0(G)\otimes\mathrm{C}^*_{\uni}(G))$, 
with $\lambda_{\uni}\colon L^1(G)\to \mathrm{C}^*_{\uni}(G)$ 
given by $\lambda_{\uni}(\omega)=(\omega\otimes\id)(W_{\uni})$ 
(see \cite{kustermans:univ}).
Thus, there is a bijective correspondence between the unitary 
representations $U$ of $G$ and the $*$-repre\-sen\-tations $\phi_U$ 
of $\mathrm{C}^*_{\uni}(G)$ on $\cH_U$, 
which is given by $\phi_U(\lambda_{\mathrm{u}}(\omega))=(\omega\otimes\id)(U)$, 
or equivalently $U = (\id\otimes\phi_U)(W_{\uni})$ 
(see \cite[Proposition 5.2]{kustermans:univ}).  
A \emph{coefficient} of $U$ is defined to be 
$f_{\omega} \in L^\infty(G)$, which is given by $f_\omega = (\id\otimes\omega)(U)$ 
for $\omega\in\IB(\cH_U)_*$. When $\omega=\omega_\eta$ is the vector functional 
associated with $\eta\in\cH_U$, we simply write $f_\eta = f_\omega$. 
The \emph{$LG$-$LG$ bimodule associated with a unitary representation $U$} 
is the bimodule $\cH:=L^2(G)\otimes\cH_U$ which is given by the tensor product
representation $W\circledT U:=W_{12}U_{13}$ of $G$ on $L^2(G)\otimes \cH$.
Namely, $\lambda(\omega) \in LG$ acts on $\cH$ from the left by 
\[
\phi_{W\circledT U}(\lambda_{\uni}(\omega))=(\omega\otimes\id\otimes\id)(W\circledT U)=U(\lambda(\omega)\otimes 1)U^*
\]
and from the right by $\hat{J}\lambda(\omega)^*\hat{J}\otimes 1$. 
Let $V$ be another unitary representation of $G$. 
Then, $U$ is said to be \emph{weakly contained} in $V$ 
(denoted by $U\preceq V$) 
if $\phi_U$ is weakly contained in $\phi_V$, i.e., 
if the identity map extends to a continuous $*$-homo\-mor\-phism $\sigma$ 
from $\phi_V(\mathrm{C}^*_{\uni}(G))$ to $\phi_U(\mathrm{C}^*_{\uni}(G))$,
which will satisfy $(\id\otimes\sigma)(V)=U$. 

\begin{prop}\label{prop:weakcontainment}
Let $U$ and $V$ be unitary representations of $G$ such that $U\preceq V$. 
Then one has ${}_{LG}(L^2(G)\otimes\cH_U)_{LG}\preceq{}_{LG}(L^2(G)\otimes\cH_V)_{LG}$.
\end{prop}
The following lemma will be used in the proof and later. The operator $T_{\hat{\varphi},\hat{\varphi}}(\xi)$ 
appearing in Lemma~\ref{lem:l2op} is simply written as $T_{\hat{\varphi}}(\xi)$.
\begin{lem}\label{lem:unitaryambi}
Let $\zeta=\xi\otimes\eta\in {}_{LG}(L^2(G)\otimes\cH_U)_{LG}$ be a simple tensor 
such that $\xi$ is $(\hat{\varphi},\hat{\varphi})$-bounded.
Then, $\zeta$ is $(\hat{\varphi},\hat{\varphi})$-bounded 
and satisfies $T_{\hat{\varphi}}(\zeta) = T_{\hat{\varphi}}(\xi) f_\eta$.
\end{lem}
\begin{proof}
We recall here the definition of the dual weight $\hat{\varphi}$. 
Let $\cI$ be the collection of $\omega \in L^1(G)$
such that there is $C>0$ satisfying
$|\omega(x^*)|\leq C \|x\varphi^{1/2}\|$ for $x\in\fn_\varphi$.
By Riesz representation theorem,
there is $\xi(\omega)\in L^2(G)$
such that $\omega(x^*)=\langle \xi(\omega),x\varphi^{1/2}\rangle$
for $x\in\fn_\varphi$.
Note that $\cI$ is a left $L^\infty(G)$-module and one has 
$\xi(a\omega)=a\xi(\omega)$ for every $a\in L^\infty(G)$ and $\omega\in \cI$.
The dual weight $\hat{\varphi}$ is defined 
in such a way that every $\omega\in\cI$ satisfies 
$\lambda(\omega)\in\fn_{\hat{\varphi}}$ and 
$\lambda(\omega)\hat{\varphi}^{1/2}=\xi(\omega)$ (\cite[Proposition 5.22]{kustermans}).
It follows that, for every $x\in\fn_{\hat{\varphi}}$, one has 
\begin{align*}
\ip{\lambda(\omega)\zeta,\zeta x^*}
 &= \ip{(\omega\otimes\id\otimes\id)(W_{12}U_{13})(\xi\otimes\eta),\xi x^* \otimes \eta}\\
 &= \ip{(\omega\otimes\id)(W(f_\eta\otimes1))\xi,\xi x^*}\\
 &= \ip{\lambda(f_\eta\omega) \xi,\xi x^*} 
 = \ip{R_{\hat{\varphi}}(\xi)\xi(f_\eta\omega),L_{\hat{\varphi}}(\xi) \hat{\varphi}^{1/2}x^*}\\
 &= \ip{T_{\hat{\varphi}}(\xi) f_\eta\lambda(\omega)\hat{\varphi}^{1/2}, \hat{\varphi}^{1/2}x^*}.
\end{align*}
Since $\lambda(\cI)$ is dense in $\fn_{\hat{\varphi}}$, this proves the lemma.
\end{proof}
\begin{proof}[Proof of Proposition \ref{prop:weakcontainment}]
Let a simple tensor $\zeta=\xi\otimes\eta \in L^2(G)\otimes\cH_U$ 
such that $\xi$ is $(\hat{\varphi},\hat{\varphi})$-bounded and $\|\xi\|=\|\eta\|=1$ 
be given. 
Since $\phi_U$ is weakly contained in $\phi_V$, there is a net $(\omega_i)_i$
of normal states on $\IB(\cH_V)$ such that $\omega_i\circ\phi_V\to\omega_\eta\circ\phi_U$ 
pointwise on $\mathrm{C}^*_{\uni}(G)$. 
Then, one has $f_{\omega_i}\to f_\omega$ ultraweakly, since 
for every $\xi,\eta\in L^2(G)$, one has
\[
\ip{ f_\omega \xi,\eta } = (\omega_{\xi,\eta}\otimes(\omega\circ\phi_U))(W_{\uni}) 
 = \lim_i (\omega_{\xi,\eta}\otimes(\omega_i\circ\phi_V))(W_{\uni}) 
 = \lim_i \ip{ f_{\omega_i} \xi,\eta }.
\]
Here note that $ (\omega_{\xi,\eta}\otimes\id)(W_{\uni}) \in \mathrm{C}^*_{\uni}(G)$.
We may assume that $\omega_i = \sum_{j=1}^{n(i)} \omega_{\eta_{i,j}}$ 
for some $\eta_{i,j}\in\cH_V$. 
Let $\zeta_{i,j}=\xi\otimes\eta_{i,j}\in L^2(G)\otimes\cH_V$. 
Then, by Lemmas~\ref{lem:l2op} and \ref{lem:unitaryambi}, 
for every $a,x\in\fn_{\hat{\varphi}}$, one has 
\begin{align*}
\ip{ a\zeta,\zeta x^*}
 &= \ip{T_{\hat{\varphi}}(\zeta) a\hat{\varphi}^{1/2}, \hat{\varphi}^{1/2}x^*}
   = \ip{T_{\hat{\varphi}}(\xi)f_\eta a\hat{\varphi}^{1/2}, \hat{\varphi}^{1/2}x^*}\\
 &= \textstyle \lim_i \sum_{j=1}^{n(i)} \ip{T_{\hat{\varphi}}(\xi)f_{\eta_{i,j}} a\hat{\varphi}^{1/2}, \hat{\varphi}^{1/2}x^*}\\
 &= \textstyle \lim_i \sum_{j=1}^{n(i)}\ip{ a\zeta_{i,j},\zeta_{i,j} x^*}.
\end{align*}
Since $\max\{|\ip{a\zeta,\zeta x^*}|,\, |\sum_{j=1}^{n(i)}\ip{ a\zeta_{i,j},\zeta_{i,j} x^*}|\}
 \le \|T_{\hat{\varphi}}(\xi)\|\|a\hat{\varphi}^{1/2}\| \| \hat{\varphi}^{1/2}x^*\|$,
the above equality in fact holds for all $a,x\in LG$. 
This means that $\omega_\zeta\circ\pi_{L^2(G)\otimes\cH_U}$ is continuous 
on $\pi_{L^2(G)\otimes\cH_V}(LG\odot LG^{\opp})$.
Since such states $\omega_\zeta$ form a cyclic family, we conclude that 
$L^2(G)\otimes\cH_U \preceq L^2(G)\otimes\cH_V$.
\end{proof}

We will prove the partial converse to Proposition \ref{prop:weakcontainment}. 
For this, we have to consider the comultiplication 
$\hat{\Delta}_{\max}\colon \mathrm{C}^*_{\uni}(G) 
 \to \mathrm{C}^*_{\uni}(G)\otimes_{\max}\mathrm{C}^*_{\uni}(G)$ 
with respect to the maximal tensor product. 
Let $\pi_i\colon \mathrm{C}^*_{\uni}(G)
 \to M(\mathrm{C}^*_{\uni}(G)\otimes_{\max}\mathrm{C}^*_{\uni}(G))$ 
be the embeddings given by $\pi_1(a)=a\otimes 1$ and $\pi_2(a)=1\otimes a$. 
Then, we consider the unitary representation 
\[
X:=(\id\otimes\pi_2)(W_{\uni})(\id\otimes\pi_1)(W_{\uni})
 \in M(C_0(G) \otimes (\mathrm{C}^*_{\uni}(G)\otimes_{\max}\mathrm{C}^*_{\uni}(G))).
\]
Since the second variables of $(\id\otimes\pi_1)(W_{\uni})$ and $(\id\otimes\pi_2)(W_{\uni})$ 
commute, $X$ is indeed a unitary representation. 
We put $\hat{\Delta}_{\max}:=\phi_X$. 
Namely, $\hat{\Delta}_{\max}$ is the $*$-homo\-mor\-phism 
that satisfies $(\id\otimes\hat{\Delta}_{\max})(W_{\uni})=X$. 
The coassociativity of $\hat{\Delta}_{\max}$ follows from 
\begin{align*}
(\id\otimes(\hat{\Delta}_{\max}\otimes\id))(X)
 &=(\id\otimes(\hat{\Delta}_{\max}\otimes\id))((\id\otimes\pi_2)(W_{\uni})\cdot(\id\otimes\pi_1)(W_{\uni}))\\
 &=(\id\otimes\pi'_3)(W_{\uni})\cdot(\id\otimes\pi'_2)(W_{\uni})(\id\otimes\pi'_1)(W_{\uni})\\
 &=(\id\otimes(\id\otimes\hat{\Delta}_{\max}))(X).
\end{align*}
Here $\pi_i'\colon \mathrm{C}^*_{\uni}(G) \to M(\mathrm{C}^*_{\uni}(G)\otimes_{\max}\mathrm{C}^*_{\uni}(G)\otimes_{\max}\mathrm{C}^*_{\uni}(G))$ denote the obvious embeddings. 
Moreover, 
for the quotient map 
$q\colon \mathrm{C}^*_{\uni}(G)\otimes_{\max}\mathrm{C}^*_{\uni}(G) 
 \to \mathrm{C}^*_{\uni}(G)\otimes\mathrm{C}^*_{\uni}(G)$,
the map $q\circ \hat{\Delta}_{\max}$ is equal to the usual 
comultiplication $\hat{\Delta}_{\uni}$ on $\mathrm{C}^*_{\uni}(G)$ (see \cite[p.311]{kustermans}). 

Let $\cH$ be an $LG$-$LG$ bimodule with the $*$-repre\-sen\-tation 
$\pi_{\cH}\colon LG\odot LG^{\opp} \to \IB(\cH)$. We define 
the \emph{unitary representation $U_{\cH}$ associated with} $\cH$ to be 
the one given by
\[
\phi_{U_{\cH}}=\pi_{\cH}\circ(\lambda\otimes\lambda^{\opp})\circ\hat{\Delta}_{\max}
 \colon \mathrm{C}^*_{\uni}(G)\to\IB(\cH).
\]
Here we identify the 
$*$-homo\-mor\-phisms from $\mathrm{C}^*_{\uni}(G)$ with the corresponding 
representations from $L^1(G)$, and define 
$\lambda^{\opp}\colon \mathrm{C}^*_{\uni}(G)\to \mathrm{C}^*_\lambda(G)^{\opp}$ by 
$\lambda^{\opp}(\omega) := \hat{R}(\lambda(\omega))^{\opp}$, where 
$\hat{R}$ is the unitary antipode. Let us define 
$\pi_{\cH}^{(i)}\colon \mathrm{C}^*_{\uni}(G) \to \IB(\cH)$ by 
$\pi_{\cH}^{(1)}(\lambda_{\uni}(\omega))=\pi_{\cH}(\lambda(\omega)\otimes1)$ and 
$\pi_{\cH}^{(2)}(\lambda_{\uni}(\omega))=\pi_{\cH}(1\otimes\lambda^{\opp}(\omega))$.
Then, it follows from the definition that 
\[
U_{\cH}=(\id\otimes\pi_{\cH}^{(2)})(W_{\uni})(\id\otimes\pi_{\cH}^{(1)})(W_{\uni})
 \in M(C_0(G)\otimes \mathrm{C}^*(\pi_{\cH}(\mathrm{C}^*_\lambda(G)\odot \mathrm{C}^*_\lambda(G)^{\opp}))).
\]

\begin{prop}\label{prop:weakcontainment2}
If $\cH$ and $\cK$ are $LG$-$LG$ bimodules such that $\cH\preceq\cK$, then $U_{\cH}\preceq U_{\cK}$.
\end{prop}

\begin{proof}
This is obvious from the definition.
\end{proof}

The \emph{conjugation representation} $V_{\mathrm{c}}:=U_{L^2(G)}$ 
is the one that is associated with the identity bimodule of $LG$ and is given by 
$V_{\mathrm{c}}=(1\otimes K)W(1\otimes K)^*W$, where $K=\hat{J}J$. 
Indeed, one has $(\id\otimes\pi_{L^2(G)}^{(1)})(W_{\uni})=W$ and 
$(\id\otimes\pi_{L^2(G)}^{(2)})(W_{\uni})=(1\otimes K)W(1\otimes K)^*$, since
\[
(\omega\otimes\pi_{L^2(G)}^{(2)})(W_{\uni}) = \hat{J}\hat{R}(\lambda(\omega))^*\hat{J} 
 = K\lambda(\omega)K^*=(\omega\otimes\id)((1\otimes K)W(1\otimes K)^*)
\]
for every $\omega\in L^1(G)$. 
We say a locally compact quantum group $G$ is \emph{strongly inner amenable} 
(in the locally compact setting, see \cite{lp}) if the trivial representation $1$ 
is weakly contained in $V_{\mathrm{c}}$. 
This property is formally stronger than the \emph{inner amenability} as introduced in \cite{gni}. 
All inner amenable locally compact groups, strongly amenable locally compact 
quantum groups, and unimodular discrete quantum 
groups are inner amenable. Since it is irrelevant to the present work, we omit the rather 
routine proofs of ``strong amenability $\Rightarrow$ strong inner 
amenability $\Rightarrow$ inner amenability.''

Proposition~\ref{prop:weakcontainment} implies the well-known fact (\cite{bct}) that 
if $G$ is strongly amenable (i.e., $1\preceq W$), then $LG$ is amenable (i.e., 
${}_{LG}L^2(G)_{LG}\preceq {}_{LG}(L^2(G)\otimes L^2(G))_{LG}$). 
Conversely, if $LG$ is amenable, then $V_{\mathrm{c}}\preceq W$ 
by Proposition~\ref{prop:weakcontainment2}.
Hence if $G$ is moreover strongly inner amenable, then $G$ is strongly 
amenable (cf.\ \cite{lp}). It would be interesting to know whether 
every discrete quantum group is strongly inner amenable (cf.\ \cite{tomatsu}), 
and whether the property $V_{\mathrm{c}}\preceq W$ is equivalent 
to amenability of $LG$.

\begin{prop}\label{prop:reciprocity} The following hold.
\begin{enumerate}[(1)]
\item
Let $U$ be a unitary representation of $G$ on $\cH$ and 
$\cH=L^2(G)\otimes\cH_U$ be the associated $LG$-$LG$ bimodule. 
Then, the unitary representation $U_{\cH}$ associated with $\cH$ is equal to 
$V_{\mathrm{c}}\circledT U$.
In particular, if $G$ is strongly inner amenable, then $U\preceq U_{\cH}$. 
\item 
Let $\cH$ be an $LG$-$LG$ bimodule and $U_{\cH}$ be 
the associated unitary representation of $G$ on $\cH$. 
Then, the $LG$-$LG$ bimodule $L^2(G)\otimes\cH$ associated with $U_{\cH}$ 
is unitarily equivalent to ${}_{\hat{\Delta}(LG)}(\cH \otimes L^2(G))_{\hat{\Delta}(LG)}$. 
\end{enumerate}
\end{prop}
\begin{proof}
\textbf{Ad(1):} A routine computation shows 
\begin{alignat*}{2}
U_{\cH} &= (\id\otimes\pi_{\cH}^{(2)})(W_{\uni})(\id\otimes\pi_{\cH}^{(1)})(W_{\uni}) 
 &&= ((1\otimes K)W(1\otimes K)^*)_{12} \cdot (W\circledT U)\\
 &=  ((1\otimes K)W(1\otimes K)^* W)_{12}U_{13}
 &&= V_{\mathrm{c}}\circledT U.
\end{alignat*}
This proves the first assertion.
If $1\preceq V_{\mathrm{c}}$, then $U=1\circledT U \preceq V_{\mathrm{c}}\circledT U$. 

\textbf{Ad(2):} To ease the notation, write $U_i=(\id\otimes\pi_{\cH}^{(i)})(W_{\uni})$ 
and $\pi_\cH(a\otimes x^{\opp})=\pi_1(a)\pi_2(x^{\opp})$, and 
denote by $\Sigma$ the flip 
either on $L^2(G)\otimes L^2(G)$ or on $\cH\otimes L^2(G)$.
Note that $U_{\cH}=U_2U_1$ and that 
$\hat{\Delta}(a)=\Sigma W (a\otimes1)W^*\Sigma$ for $a\in LG$ (\cite[Theorem 5.17]{kustermans}).
Thus for the unitary operator  
$Y := U_2\Sigma$ from $\cH \otimes L^2(G)$ onto $L^2(G)\otimes\cH$, one has 
\begin{align*}
Y^* \pi_{L^2(G)\otimes\cH}(a\otimes1) Y 
 = \Sigma^* U_2^*U_{\cH}(a\otimes1)U_{\cH}^*U_2\Sigma 
 = \Sigma^* U_1(a\otimes1)U_1^*\Sigma 
 = (\pi_1\otimes\id)(\hat{\Delta}(a)).
\end{align*}
We abuse the notation and view $\pi_{\cH}^{(2)}$ as a $*$-homo\-mor\-phism from 
$LG$ into $\IB(\cH)$, which is given by $\pi_{\cH}^{(2)}(x)=\pi_2(\hat{R}(x)^{\opp})$.
The right action of $LG$ on $L^2(G)$ is denoted by $\rho(y^{\opp})=\hat{J}y^*\hat{J}$. 
Note that
\[
(\pi_2\otimes\rho)( x^{\opp}\otimes y^{\opp} ) = \pi_{\cH}^{(2)}(Jx^*J) \otimes \hat{J}y^*\hat{J} 
 = (\pi_{\cH}^{(2)}\otimes\id)( (J\otimes\hat{J}) (x\otimes y)^* (J\otimes\hat{J}) ).
\]
Since $(\hat{J}\otimes J) W (\hat{J}\otimes J) = W^*$ (\cite[Section 5.3]{kustermans}), it follows that 
\begin{align*}
Y^*\pi_{L^2(G)\otimes\cH}(1\otimes x^{\opp}) Y 
 &= \Sigma^*U_2^*(\hat{J}x^*\hat{J}\otimes 1)U_2\Sigma 
 = (\pi_{\cH}^{(2)}\otimes\id)(\Sigma W^*(\hat{J}x^*\hat{J}\otimes JJ)W\Sigma) \\
 &= (\pi_{\cH}^{(2)}\otimes\id)((J\otimes\hat{J})\hat{\Delta}(x^*)(J\otimes\hat{J}))
 = (\pi_2\otimes\rho)(\hat{\Delta}(x)^{\opp}).
\end{align*}
This proves the assertion.
\end{proof}

Recall from \cite{dfsw} that a unitary representation $U$ is said to 
be \emph{mixing} if the coefficient $f_\omega=(\id\otimes\omega)(U)$ belongs 
to $C_0(G)$ for every $\omega\in\IB(\cH_U)_*$. 

\begin{prop}\label{prop:mixingreptobimod}
If $U$ is a mixing unitary representation of a locally compact quantum group $G$,  
then the $LG$-$LG$ bimodule $L^2(G)\otimes\cH_U$ is strictly mixing. 
Conversely, if $\cH$ is a strictly mixing $LG$-$LG$ bimodule, 
then the unitary representation $U_{\cH}$ is mixing. 
\end{prop}

\begin{proof}
Let $\omega \in\cI$ (see Proof of  Lemma~\ref{lem:unitaryambi}) 
be an element which 
is analytic with respect to $t\mapsto \rho_t(\omega) := \omega(\delta^{-it}\tau_{-t}(\,\cdot\,))$ 
(see \cite[5.22]{kustermans} or \cite[8.7]{kv} for the notation).
Then, the vector $\xi:=\lambda(\omega)\hat{\varphi}^{1/2}$ is bounded with 
$L_{\hat{\varphi}}(\xi)=\lambda(\omega) \in \mathrm{C}^*_\lambda(G)$ and 
$R_{\hat{\varphi}}(\xi)=\hat{J}\lambda(\rho_{i/2}(\omega))^*\hat{J} \in 
\hat{J}\mathrm{C}^*_\lambda(G)\hat{J}$. 
Hence, by Lemma~\ref{lem:unitaryambi}, the vector 
$\zeta := \xi \otimes \eta$ is bounded for every $\eta\in\cH_U$ 
and satisfies 
\[
T_{\hat{\varphi}}(\zeta) = L_{\hat{\varphi}}(\xi)^*R_{\hat{\varphi}}(\xi) f_\eta 
 \in \mathrm{C}^*_\lambda(G)\cdot \hat{J}\mathrm{C}^*_\lambda(G)\hat{J}\cdot C_0(G) 
 \subset \IK(L^2(G)),
\]
where the last inclusion is by \cite[Lemma 5.5]{bsv}.
Since such $\zeta$'s have a dense span, 
this proves that  $L^2(G)\otimes\cH_U$ is strictly mixing. 

For the converse, it suffices to show $(\id\otimes\omega_{a\eta ,\eta x^*})(U_{\cH}) \in C_0(G)$ 
for every $\eta\in\cH_{\mathrm{mix}}$ and $a,x\in n_\varphi$.
Let us fix $\xi\in L^2(G)$ and we will compute 
\[
\ip{(\id\otimes\omega_{a\eta, \eta x^*})(U_{\cH})\xi,\xi}
 = \ip{(\id\otimes\pi_{\cH}^{(1)})(W_{\uni})(\xi\otimes a\eta),
  (\id\otimes\pi_{\cH}^{(2)})(W_{\uni}^*)(\xi\otimes \eta x^*)}.
\]
Let $\zeta \in L^2(G)$ be given and consider
$L_\zeta\colon \cH\ni \eta'\mapsto \zeta\otimes\eta' \in L^2(G)\otimes \cH$.
Then, 
\begin{align*}
L_\zeta^*(\id\otimes\pi_{\cH}^{(2)})(W_{\uni}^*)(\xi\otimes \eta x^*)
 &= \pi_{\cH}^{(2)}((\omega_{\xi,\zeta}\otimes\id)(W_{\uni}^*))\eta x^* 
 = \eta x^*\hat{R}((\omega_{\xi,\zeta}\otimes\id)(W^*)) \\
 &= L_{\hat{\varphi}}(\eta) \hat{J}\hat{R}((\omega_{\xi,\zeta}\otimes\id)(W^*))^*\hat{J} \hat{\varphi}^{1/2}x^* \\
 &= L_{\hat{\varphi}}(\eta)  K (\omega_{\xi,\zeta}\otimes\id)(W^*) K^*\hat{\varphi}^{1/2}x^* \\
 &= L_\zeta^* (1\otimes L_{\hat{\varphi}}(\eta) K) W^*(1\otimes K)^* (\xi\otimes\hat{\varphi}^{1/2}x^*).
\end{align*}
Since $\zeta\in L^2(G)$ was arbitrary, it follows that 
\begin{align*}
(\id\otimes\pi_{\cH}^{(2)})(W_{\uni}^*)(\xi\otimes \eta x^*)
 &= (1\otimes L_{\hat{\varphi}}(\eta) K) W^*(1\otimes K)^* (\xi\otimes\hat{\varphi}^{1/2}x^*).
\intertext{A similar but much easier computation shows }
(\id\otimes\pi_{\cH}^{(1)})(W_{\uni})(\xi\otimes a\eta) 
 &= (1\otimes R_{\hat{\varphi}}(\eta)) W(\xi\otimes a \hat{\varphi}^{1/2}).
\end{align*}
Therefore, by Lemma~\ref{lem:l2op}, 
\[
\ip{(\id\otimes\omega_{a \eta, \eta x^*})(U_{\cH})\xi,\xi} 
 =  \ip{(1\otimes K) W (1\otimes K^* T_{\hat{\varphi}}(\eta)) W (\xi\otimes a\hat{\varphi}^{1/2}), (\xi\otimes \hat{\varphi}^{1/2}x^*)}.
\]
Since $\xi\in L^2(G)$ was arbitrary, this implies 
\[
(\id\otimes\omega_{a\eta, \eta x^*})(U_{\cH})
 = (\id\otimes\omega_{\hat{\varphi}^{1/2}x^*,a \hat{\varphi}^{1/2}})
  ((1\otimes K) W (1\otimes K^* T_{\hat{\varphi}}(\eta)) W).
\]
Since $T_{\hat{\varphi}}(\eta)$ is compact, we conclude that 
$(\id\otimes\omega_{a\eta, \eta x^*})(U_{\cH})\in 
\mathrm{C}^*\{(\id\otimes\omega)(W) : \omega \} = C_0(G)$. 
This finishes the proof of the converse. 
\end{proof}

Recall from \cite{dfsw} that a locally compact quantum group is said to have the 
\emph{Haagerup property} if the trivial unitary representation $1$ is weakly contained 
in a mixing unitary representation. 
The following theorem extends the same result for discrete quantum groups 
in \cite[Theorems 7.4 and 7.7]{dfsw}.

\begin{thm}
Let $G$ be a locally compact quantum group. 
If $G$ is has the Haagerup property, then $LG$ has the HAP. 
Conversely, if $G$ is strongly inner amenable and $LG$ has the HAP, then 
$G$ has the Haagerup property.
\end{thm}
\begin{proof}
First suppose that $G$ has the Haagerup property, i.e., there is a mixing unitary representation 
$U$ such that $1\preceq U$. Then, ${}_{LG}L^2(G)_{LG}\preceq {}_{LG}(L^2(G)\otimes\cH_U)_{LG}$ 
by Proposition~\ref{prop:weakcontainment}, but the latter bimodule is 
strictly mixing by Proposition~\ref{prop:mixingreptobimod}.
Now, Theorem~\ref{thm:bihap} applies and we conclude that $LG$ has the HAP. 
Conversely, suppose that $LG$ has the HAP. Then by Theorem~\ref{thm:bihap} 
there is a strictly mixing bimodule $\cH$ such that ${}_{LG}L^2(G)_{LG}\preceq {}_{LG}\cH_{LG}$.
By Proposition~\ref{prop:mixingreptobimod}, the associated unitary representation 
$U_{\cH}$ is mixing and, by Proposition~\ref{prop:weakcontainment2}, 
it satisfies $V_{\mathrm{c}}\preceq U_{\cH}$.  
Hence, if $G$ is moreover strongly inner amenable, then the mixing representation $U_{\cH}$ 
weakly contains the trivial representation.
\end{proof}

While it may be true that the HAP (resp.\ amenability) of $LG$ implies 
the Haagerup property (resp.\ amenability) of $G$ for discrete quantum groups, 
this need not be true for general locally compact quantum groups. 
For example, a simple connected higher rank Lie group, such as $\mathrm{SL}(3,\IR)$, 
has a type $\mathrm{I}$ and hence amenable group von Neumann algebra, 
but it does not have the Haagerup 
property because of Kazhdan's property (T) (see \cite{ccjjv}).

\end{document}